\documentclass{amsart}
\usepackage{amsfonts}

\newtheorem{theorem}{Theorem}
\theoremstyle{plain}

\newtheorem{corollary}{Corollary}

\newtheorem{lemma}{Lemma}

\newtheorem{proposition}{Proposition}
\newtheorem{remark}{Remark}

\numberwithin{equation}{section}
\input{tcilatex}

\begin{document}
\title[Chebyshev polynomials]{On probabilistic aspects of Chebyshev
polynomials }
\author{Pawe\l\ J. Szab\l owski}
\address{Department of Mathematics and Information Sciences,\\
Warsaw University of Technology\\
ul Koszykowa 75, 00-662 Warsaw, Poland}
\email{pawel.szablowski@gmail.com}
\date{August 2015}

\begin{abstract}
The main goal of this note is to provide new, mostly multidimensional
densities, compactly supported and list many of its properties that enable
effective calculations. The idea of obtaining such densities is firstly to
build some one-dimensional densities depending on many parameters and then
treat the constructed in this way distributions as conditional ones. Then of
course by imposing certain distribution on the parameters and multiplying
the two distributions we arrive at new multivariate distribution. To enable
effective calculations, we utilize nice, simple and widely known properties
of Chebyshev polynomials. Thus, in particular, the one-dimensional
distribution mentioned above will have a form of arcsine distribution
multiplied by some rational function. The fact that we use Chebyshev
polynomials allows us to calculate all moments of this one-dimensional
distribution as well as to find a family of polynomials that are orthogonal
with respect to this distribution.
\end{abstract}

\maketitle

\section{Introduction}

The purpose of this note is to recall known and establish new one and
multi-dimensional densities related to Chebyshev polynomials. As stated in
the abstract we construct multidimensional distribution by finding a
one-dimensional distribution defined with the help of many parameters and
then treat this density as the conditional one. Multiplying it by some
multi-dimensional density makes new multi-dimensional density.

The point is that this one-dimensional density, that we construct, is simple
yet versatile. It has a form of the product of the density that makes
Chebyshev polynomials orthogonal (i.e. arcsine distribution) times some
rational function of one or more variables.

By considering densities analyzed in the paper i.e. given by (\ref{fnT}), we
are able to expand the ratio of such a density and the arcsine one in the
Fourier series of Chebyshev polynomials. Obtained so expansion is in line
with ideas expressed in \cite{Szablowski2010(1)}. This simple form allows us
to calculate many important characterizations of the analyzed distribution
such as moments, consequently cumulants or to construct polynomials that are
orthogonal with respect to this one-dimensional distribution.

One has to note that the constructed one-dimensional density is of interest
by itself.

As plots presented on Figure \ref{fig1} on the page \pageref{fig1} indicate,
even for the case of $n\allowbreak =\allowbreak 2$ (i.e. $2$ parameters), we
deal with very versatile family of densities. It  contains distributions
that are compactly supported, symmetric or highly skewed, with or without
modes  between $-1$ and $1.$ Moreover, this family is indexed by only two
parameters that are easily estimated by the moment method (compare formulae
given in Corrolaries \ref{wyl} and \ref{cases}). Hence fitting different
sets of data coming from one source, can require just one model  with only
two parameters. Or more than two different sets of data can come from one
source (be described by one model with different values of  parameters). As
the Remark \ref{f4} shows, similar general properties, however including
bimodal case, have densities  $f_{4}$ (given by (\ref{f44})) with $4$
parameters. Corrolaries \ref{wyl} and \ref{mom} show how to estimate those
parameters by the moment method.

Technically, this paper is complementary to the paper \cite{Szab-Kes} where
we consider generalizations of the so-called Kesten--McKay distributions.
That paper differs from the one considered herein that there we consider
expansions of one-dimensional distribution in Chebyshev polynomials of the
second kind while in the present paper we consider expansions in Chebyshev
polynomials of the first kind. One can notice that the difference in
difficulty of calculations is big. What is more interesting the densities
considered in this paper are undefined at the ends of its supports while
densities considered in \cite{Szab-Kes} assume value 0 at the ends of their
supports.

The paper is organized as follows. In Subsection \ref{defff} we fix notation
and recall definition and some well-known properties of the Chebyshev
polynomials. We also present some "easy to get" one and two-dimensional
compactly supported distributions related to the Chebyshev polynomials of
the first kind. In the next Section \ref{auxx} we recall some auxiliary
results that will be used in the sequel. In Section \ref{gest} we present
our main results i.e. we expand ratio of the considered one-dimensional
density and the arcsine density in the Fourier series in Chebyshev
polynomials of the first kind. This enables us to find all moments of the
considered distribution and last but not least we find a family of
polynomials that are orthogonal with respect to this distribution.

In Subsection \ref{zesp} we consider the case of even number of parameters
and parameters forming conjugate pairs. This is done to expose analogies
with the probabilistic interpretation of chain of distributions related to
Askey--Wilson family as done in \cite{SzablAW}. This procedure turns out to
lead to even more versatile, yet much more complicated families of
multidimensional distribution.

Section \ref{dow} contains longer proofs.

Notice that in many cases such Fourier expansions mentioned above in the
case of 2 dimensional distributions would provide examples of Lancaster
expansions as described in the series of papers \cite{Lancaster58}, \cite%
{Lancaster63(1)}, \cite{Lancaster63(2)}.

\subsection{Notation and definitions\label{defff}}

To fix notation we will denote by $\left\{ T_{n}(x)\right\} _{n\geq 0}$
Chebyshev polynomials of the first kind, while $\left\{ U_{n}(x)\right\}
_{n\geq 0}$ will denote Chebyshev polynomials of the second kind. Recall
that these two families satisfy the same 3-term recurrence 
\begin{equation}
T_{n+1}(x)=2xT_{n}(x)-T_{n-1}(x),  \label{3r}
\end{equation}%
with different initial conditions. Namely $U_{0}(x)\allowbreak =\allowbreak
T_{0}(x)\allowbreak =\allowbreak 1,$ $T_{1}(x)\allowbreak =\allowbreak x$
and $U_{1}(x)\allowbreak =\allowbreak 2x.$ Recall also that polynomials $%
\left\{ T_{n}\right\} $ are orthogonal with respect to the density 
\begin{equation*}
f_{C}(x)\allowbreak =\allowbreak 1/(\pi \sqrt{1-x^{2}})I_{(-1,1)}(x),
\end{equation*}%
called arcsine distribution, while polynomials $\left\{ U_{n}\right\} $ are
orthogonal with respect to the density 
\begin{equation*}
f_{W}(x)\allowbreak =\allowbreak \frac{2}{\pi }\sqrt{1-x^{2}}I_{[-1,1]}(x),
\end{equation*}%
called semicircle or Wigner distribution. Here and below we used denotation 
\begin{equation*}
I_{A}(x)=\left\{ 
\begin{array}{ccc}
1 & if & x\in A \\ 
0 & if & otherwise%
\end{array}%
\right. .
\end{equation*}

Let us recall that 
\begin{equation}
\int_{-1}^{1}T_{i}(x)T_{j}(x)f_{C}(x)dx\allowbreak =\allowbreak \left\{ 
\begin{array}{ccc}
0 & if & i\neq j \\ 
1/2 & if & i=j\neq 0 \\ 
1 & if & i=j=0%
\end{array}%
\right. ,  \label{ortT}
\end{equation}%
and 
\begin{equation}
\int_{-1}^{1}U_{i}(x)U_{j}(x)f_{W}(x)dx\allowbreak =\allowbreak \left\{ 
\begin{array}{ccc}
0 & if & i\neq j \\ 
1 & if & i=j%
\end{array}%
\right. .  \label{ortU}
\end{equation}

In the sequel we will use alternatively $\mathbf{a}_{n}\mathbf{\allowbreak
\allowbreak }$ and $(a_{1},\ldots ,a_{n})^{T}$ or sometimes we will drop
dependence on $\mathbf{a}_{n}$ if it will be obvious.

Basically, we will be mostly concerned with the analysis of the following
densities:%
\begin{equation}
f_{nT}(x|\mathbf{a}_{n})=B_{n}(\mathbf{a}_{n})f_{C}(x)\prod_{j=1}^{n}\varphi
_{T}(x|a_{i}).  \label{fnT}
\end{equation}%
Here $\varphi _{T}$ denotes the generating function of polynomials $T_{n},$
i.e. 
\begin{equation}
\varphi _{T}(x|a)\allowbreak =\allowbreak \sum_{j=0}^{\infty
}a^{j}T_{j}(x)=(1-ax)/(1+a^{2}-2ax),  \label{fTg}
\end{equation}%
defined for $\left\vert a\right\vert <1$ and $B_{n}$ is suitable constant.
Obviously we have $B_{1}\allowbreak =\allowbreak 1.$

We will assume that $\left\vert a_{i}\right\vert <1,$ $i\allowbreak
=\allowbreak 1,\ldots ,n.$

It is easy to notice that $\varphi _{T}$ is positive for all $\left\vert
a\right\vert ,\left\vert x\right\vert <1.$ Especially interesting will be
the cases with even $n$ and parameters $a_{i}$ forming conjugate pairs.

\section{Auxiliary results\label{auxx}}

Let us denote for simplicity: 
\begin{equation*}
w(x,y|\rho )=(1-\rho ^{2})^{2}-4xy\rho (1+\rho ^{2})+4\rho ^{2}(x^{2}+y^{2}).
\end{equation*}%
Further let us recall the so called Poisson--Mehler expansion formula
presented and proved in e.g. \cite{bressoud}, \cite{IA}, \cite{SzabP-M}
considered for $q=0$ gives:%
\begin{equation}
g_{U}(x,y|\rho )\overset{df}{=}\frac{1-\rho ^{2}}{w(x,y|\rho )}\allowbreak
=\allowbreak \sum_{j\geq 0}\rho ^{j}U_{j}(x)U_{j}(y),  \label{expU}
\end{equation}%
convergent almost everywhere for $\left\vert x\right\vert ,\left\vert
y\right\vert \,\leq 1$ $\left\vert \rho \right\vert <1$ since $%
\max_{\left\vert x\right\vert \leq 1}\left\vert U_{n}(x)\right\vert
\allowbreak =\allowbreak (n+1)$ and the series $\sum_{j\geq 0}\rho
^{j}(j+1)^{2}$ is convergent.

\begin{proposition}
\label{aux}$\forall \left\vert x\right\vert ,\left\vert y\right\vert \leq
1,\left\vert \rho \right\vert <1:$%
\begin{gather}
g_{T}(x,y|\rho )\overset{df}{=}\frac{1-\rho ^{2}-xy\rho (3+\rho ^{2})+2\rho
^{2}(x^{2}+y^{2})}{w(x,y|\rho )}=\sum_{j=0}^{\infty }\rho
^{j}T_{j}(x)T_{j}(y)\geq 0,  \label{expT} \\
g_{UT}(x,y|\rho )\overset{df}{=}\frac{1-\rho ^{2}-2\rho xy+2\rho ^{2}y^{2}}{%
w(x,y|\rho )}=\sum_{j=0}^{\infty }\rho ^{j}U_{j}(x)T_{j}(y),  \label{expUT}
\\
\frac{2\rho ^{2}(y-\rho x)}{w(x,y|\rho )}=\sum_{j=0}^{\infty }\rho
^{j}U_{j}(x)U_{j+1}(y),  \label{n_n-1} \\
\frac{(y-\rho x)(1+\rho ^{2}-2\rho xy)}{w(x,y,\rho )}=\sum_{j=0}^{\infty
}\rho ^{j}T_{j}(x)T_{j+1}(y),  \label{Tn_n-1} \\
\frac{4x^{2}-1-\rho }{(1-\rho )((1+\rho )^{2}-4\rho x^{2})}%
=\sum_{j=0}^{\infty }\rho ^{j}U_{j}(x)U_{j+2}(x),  \label{n_n-2} \\
\frac{2(x^{2}+y^{2}-2\rho xy)}{w(x,y|\rho )}=x\sum_{n\geq 0}^{\infty }\rho
^{n}U_{n+1}(x)U_{n}(y)+y\sum_{n\geq 0}^{\infty }\rho
^{n}U_{n+1}(y)U_{n}(x)\geq 0,  \label{A}
\end{gather}%
convergent almost everywhere for $\left\vert x\right\vert ,\left\vert
y\right\vert \,\leq 1$ $\left\vert \rho \right\vert <1.$
\end{proposition}

\begin{proof}
Is shifted to Section \ref{dow}
\end{proof}

As a corollary we get the following fact:

\begin{corollary}
$\forall \left\vert \rho \right\vert <1,$ $\left\vert x\right\vert \leq 1:$%
\begin{gather}
\sum_{j=0}^{\infty }\rho ^{2j}T_{2j}(x)\allowbreak =\allowbreak \frac{1+\rho
^{2}-2x^{2}\rho ^{2}}{(1+\rho ^{2}-2x\rho )(1+\rho ^{2}+2x\rho )}\geq 0,
\label{parz} \\
\sum_{j=0}^{\infty }\rho ^{2j+1}T_{2j+1}(x)=\frac{x\rho (1-\rho ^{2})}{%
(1+\rho ^{2}-2x\rho )(1+\rho ^{2}+2x\rho )},  \label{nparz} \\
\sum_{j\geq 0}\rho ^{j}U_{j}(x)T_{j}(x)=\frac{(1+\rho ^{2}-2\rho x^{2})}{%
(1-\rho )((1+\rho )^{2}-4\rho x^{2})}\geq 0,  \label{TU} \\
\sum_{j\geq 0}\rho ^{4j}T_{4j}(x)=\frac{1-\rho ^{4}+8\rho ^{4}x^{2}(1-x^{2})%
}{((1+\rho ^{2})^{2}-4\rho ^{2}x^{2})((1-\rho ^{2})^{2}+4\rho ^{2}x^{2})}%
\geq 0.  \label{t4}
\end{gather}
\end{corollary}

\begin{proof}
(\ref{parz}) and (\ref{nparz}). First we take $x\allowbreak =\allowbreak y$
in (\ref{expT}), secondly we use the formula $T_{2n}(x)\allowbreak
=\allowbreak 2T_{n}^{2}(x)-1$ and change $\rho $ to $\rho ^{2}.$ (\ref{TU}).
We put $y=x$ in (\ref{expUT}). To get (\ref{t4}) we use the fact that $%
T_{4j}(x)\allowbreak =\allowbreak T_{2j}(T_{2}(x))$ hence, in (\ref{parz})
we set $x->2x^{2}-1$ and $\rho ^{2}->\rho ^{4}.$
\end{proof}

Recall also that from the well known identities: 
\begin{eqnarray*}
T_{n}(x)\allowbreak &=&\allowbreak \frac{1}{2}(U_{n}(x)-U_{n-2}(x), \\
U_{2k}(x)
&=&2\sum_{j=0}^{k}T_{2j}(x)+1,~U_{2k+1}(x)=2\sum_{j=0}^{k}T_{2j+1}(x),
\end{eqnarray*}%
valid for $n,k\geq 0$, it follows that 
\begin{eqnarray}
\int_{-1}^{1}U_{n}(x)f_{C}(x)dx &=&\left\{ 
\begin{array}{ccc}
1 & if & n\allowbreak =\allowbreak 0\text{ or is even} \\ 
0 & if & n\text{ is odd}%
\end{array}%
\right. ,  \label{iu} \\
\int_{-1}^{1}T_{n}(x)f_{W}(x)dx &=&\left\{ 
\begin{array}{ccc}
1 & if & n=0 \\ 
0 & if & n=1,3,... \\ 
-1/2 & if & n=2%
\end{array}%
\right. .  \label{it}
\end{eqnarray}%
Following the results of the above mentioned Proposition and Corollary, we
can define four two-dimensional densities depending on one parameter,
defined on the square $[-1,1]x[-1,1]$. Some of these densities are new. 
\begin{eqnarray}
f_{UU}(x,y) &=&f_{W}(x)f_{W}(y)\frac{1-\rho ^{2}}{w(x,y|\rho )},  \label{_f1}
\\
f_{TT}(x,y)\allowbreak &=&\allowbreak f_{C}(x)f_{C}(y)\frac{(1-\rho
^{2})-xy\rho (3+\rho ^{2})+2\rho ^{2}(x^{2}+y^{2})}{w(x,y|\rho )},
\label{_f2} \\
f_{UT}(x,y) &=&f_{W}(x)f_{C}(y)\frac{4(1-\rho ^{2})-8\rho xy+4\rho
^{2}(x^{2}+y^{2})}{(2-\rho ^{2})w(x,y|\rho )},  \label{_f3} \\
f_{2UU} &=&f_{W}(x)f_{W}(y)\frac{2(x^{2}+y^{2}-2\rho xy)}{w(x,y|\rho )}.
\label{_f4}
\end{eqnarray}

The fact that (\ref{_f1}), (\ref{_f2}) and (\ref{_f4}) are densities does
not require justification. To show that (\ref{_f3}) is can be justified in
the following way. Firstly notice that it is non-negative for $\left\vert
x\right\vert ,\left\vert y\right\vert \leq 1.$ We do so by calculating its
discriminant which is equal to $64\rho ^{2}y^{2}-16\rho ^{2}(4(1-\rho
^{2})+4\rho ^{2}y^{2})=\allowbreak -64\rho ^{2}\left( 1-\rho ^{2}\right)
\left( 1-y^{2}\right) \leq 0.\allowbreak $ To see that it integrates to $1,$
let us notice, that 
\begin{equation*}
\frac{2(1-\rho ^{2})-4\rho xy+2\rho ^{2}(x^{2}+y^{2})}{w(x,y|\rho )}%
=\sum_{j=0}^{\infty }\rho ^{j}U_{j}(x)T_{j}(y)+\sum_{j=0}^{\infty }\rho
^{j}U_{j}(y)T_{j}(x).
\end{equation*}%
Now $\int_{-1}^{1}\int_{-1}^{1}f_{C}(x)f_{W}(y)(\sum_{j=0}^{\infty }\rho
^{j}U_{j}(y)T_{j}(x))dxdy\allowbreak =\allowbreak 1$ \newline
while $\int_{-1}^{1}\int_{-1}^{1}f_{C}(x)f_{W}(y)(\sum_{j=0}^{\infty }\rho
^{j}U_{j}(x)T_{j}(y))dxdy\allowbreak =\allowbreak -\rho ^{2}/2$ by (\ref{it}%
).

Notice that for the densities $f_{UU}$ and $f_{TT}$ we get Lancaster
expansions of these densities for free.

Notice also that marginal densities of (\ref{_f1}),...,(\ref{_f4}) are
respectively $f_{W}$ and $f_{W},$ $f_{C}$ and $f_{C},$ $f_{W}$ and $f_{T}$
and finally again $f_{W}$ and $f_{W}.$

We obtain also some one parameter $1$ dimensional, symmetric distributions
related to distributions with densities $f_{C}$ and $f_{W}$ by making use of
(\ref{parz}), (\ref{TU}) and (\ref{t4}).

\begin{eqnarray*}
f_{1}(x|\rho )\allowbreak &=&\allowbreak f_{C}(x)\frac{1+\rho
^{2}-2x^{2}\rho ^{2}}{(1+\rho ^{2}-2x\rho )(1+\rho ^{2}+2x\rho )}, \\
f_{2C}(x|\rho ) &=&f_{C}(x)\frac{(1+\rho -2\rho x^{2})}{((1+\rho )^{2}-4\rho
x^{2})}, \\
f_{2W}(x|\rho ) &=&f_{W}(x)\frac{2(1+\rho ^{2}-2\rho x^{2})}{(2-\rho
)((1+\rho )^{2}-4\rho x^{2})}, \\
f_{3}(x|\rho ) &=&f_{C}(x)\frac{1-\rho ^{4}+8\rho ^{4}x^{2}(1-x^{2})}{%
((1+\rho ^{2})^{2}-4\rho ^{2}x^{2})((1-\rho ^{2})^{2}+4\rho ^{2}x^{2})}.
\end{eqnarray*}%
Justification is needed only for $f_{2C}$ and $f_{2W}$. The fact that these
functions integrate to $1$ follows the fact that 
\begin{equation*}
\int_{-1}^{1}U_{n}(x)T_{n}(x)f_{C}(x)dx=1,
\end{equation*}%
for $n\geq 0$ and 
\begin{equation*}
\int_{-1}^{1}U_{n}(x)T_{n}(x)f_{W}(x)dx=\left\{ 
\begin{array}{ccc}
1 & if & n=0 \\ 
1/2 & if & n\geq 1%
\end{array}%
\right. .
\end{equation*}

\section{Densities\label{gest}}

\begin{theorem}
\label{density}For $\forall n\geq 2$ and parameters $\mathbf{a}_{n}$
mutually different we get 
\begin{equation*}
B_{n}(\mathbf{a}_{n})=2^{n}/\sum_{k=0}^{n}b_{n,k},
\end{equation*}%
where we denoted for $k\allowbreak =\allowbreak 0$ : $b_{n,0}\allowbreak
=\allowbreak 1$ while for $k\allowbreak =\allowbreak 1,\ldots ,n$ the
following functions: 
\begin{equation}
b_{n,k}(\mathbf{a}_{n})\allowbreak =\allowbreak \prod_{j=1,j\neq k}^{n}\frac{%
\gamma (a_{k},a_{j})}{(a_{k}-a_{j})(1-a_{k}a_{j})},  \label{bb}
\end{equation}%
with $\gamma (a,b)\allowbreak =\allowbreak 2a-b-a^{2}b$ and $k\allowbreak
=\allowbreak 1,...,n$. Besides, we have an expansion:%
\begin{equation}
f_{nT}(x|\mathbf{a}_{n})=f_{C}(x)\sum_{j=0}^{\infty }t_{j}(\mathbf{a}%
_{n})T_{j}(x),  \label{expan1}
\end{equation}%
where $t_{0}(\mathbf{a}_{n})\allowbreak =\allowbreak 1$ and 
\begin{equation}
t_{j}(\mathbf{a}_{n})\allowbreak =\allowbreak
2\sum_{k=1}^{n}b_{n,k}a_{k}^{j}/\sum_{k=0}^{n}b_{n,k}  \label{tt}
\end{equation}%
for $j\geq 1.$
\end{theorem}

\begin{proof}
Is shifted to Section \ref{dow}.
\end{proof}

\begin{remark}
For the sake of completeness we set $\allowbreak b_{1,1}(a)\allowbreak
=\allowbreak 1.$
\end{remark}

The proof of the Theorem \ref{density} is based on the following lemma
concerning symmetric rational functions that has a value of its own

\begin{lemma}
\label{num}i) Let $n\geq 1.$ For different $b_{i},$ $i\allowbreak
=\allowbreak 1,\ldots ,n$ we have: 
\begin{equation*}
\prod_{i=1}^{n}\frac{x-a_{i}}{x-b_{i}}\allowbreak =\allowbreak
1+\sum_{k=1}^{n}\frac{\prod_{i=1}^{n}(b_{k}-a_{i})}{(x-b_{k})\prod_{i=1,i%
\neq k}^{n}(b_{k}-b_{i})},
\end{equation*}%
Let functions $b_{n,k}$ be defined by (\ref{bb}) for $n\geq k\geq 0.$ Assume
that all $a_{i},i\allowbreak =\allowbreak 1,\ldots n$ are different.

Then ii) 
\begin{equation}
\prod_{j=1}^{n}\frac{(1-a_{j}x)}{(1+a_{j}^{2}-2a_{j}x)}=\frac{1}{2^{n}}+%
\frac{1}{2^{n}}\sum_{k=1}^{n}\frac{(1-a_{k}^{2})}{(1+a_{k}^{2}-2a_{k}x)}%
b_{n,k}.  \label{dec}
\end{equation}

iii) $b_{m,k}\frac{\gamma (a_{k},a_{m+1})}{(a_{k}-a_{m+1})(1-a_{k}a_{m+1})}%
\allowbreak =\allowbreak b_{m+1,k},$

iv)%
\begin{equation*}
b_{m+1,m+1}=1+\sum_{k=1}^{m}b_{m,k}\frac{a_{m+1}(1-a_{k}^{2})}{%
(a_{m+1}-a_{k})(1-a_{k}a_{m+1})}\allowbreak .
\end{equation*}
\end{lemma}

\begin{proof}
Proof is shifted to Section \ref{dow}.
\end{proof}

To describe briefly constants $B_{n}(\mathbf{a}_{n})$ let us define the
following symmetric functions: 
\begin{eqnarray}
S_{k}^{(n)}(\mathbf{a}_{n}) &=&\sum_{0\leq j_{1}<j_{2}<\ldots ,<j_{k}\leq
n}\prod_{m=1}^{m}a_{j_{m}}  \label{sym} \\
\Delta _{m}^{(n)}(\mathbf{a}_{n}) &=&\sum_{\substack{ 0\leq j_{1},\ldots
j_{n-1}\leq m  \\ j_{1}+\ldots +j_{n-1}\leq m}}a_{1}^{j_{1}}a_{2}^{j_{2}}%
\ldots a_{n-1}^{j_{n-1}}a_{n}^{m-(j_{1}+\ldots +j_{n-1})},  \label{delta} \\
P_{n}(\mathbf{a}_{n})\allowbreak &=&\allowbreak \prod_{1\leq i<j\leq
n}(1-a_{i}a_{j}).  \label{pp}
\end{eqnarray}
As a corollary we have:

\begin{corollary}
\label{wyl}We have $B_{1}(a)\allowbreak =\allowbreak 1.$Using Mathematica we
obtain: $B_{2}(a_{1},a_{2})\allowbreak =\allowbreak \frac{2(1-a_{1}a_{2})}{%
(2-a_{1}a_{2})}\allowbreak =\allowbreak \frac{2P_{2}(a_{1},a_{2})}{%
P_{2}(a_{1},a_{2})+1}$, $B_{3}(a_{1},a_{2},a_{3})\allowbreak =\allowbreak 
\frac{4P_{3}(\mathbf{a}_{3})}{P_{3}(\mathbf{a}_{3})+3-S_{2}(\mathbf{a}_{3})}$%
, \newline
$B_{4}(\mathbf{a}_{4})\allowbreak =\allowbreak \frac{8P_{4}(\mathbf{a}_{4})}{%
P_{4}(\mathbf{a}_{4})+7-3S_{2}(\mathbf{a}_{4})+S_{3}(\mathbf{a}_{4})(S_{1}(%
\mathbf{a}_{4})-S_{3}(\mathbf{a}_{4}))-S_{4}(\mathbf{a}_{4})(4-2S_{2}(%
\mathbf{a}_{4})+3S_{4}(\mathbf{a}_{4}))}.$
\end{corollary}

Following expansion (\ref{expan1}), we can calculate all moments of the
distribution with the density $f_{nT}.$

\begin{corollary}
\label{mom}%
\begin{eqnarray*}
\int_{-1}^{1}x^{2k}f_{nT}(x|\mathbf{a}_{n})dx &=&2^{-2k}\sum_{j=0}^{k}\binom{%
2k}{k-j}t_{2j}(\mathbf{a}_{n}), \\
\int_{-1}^{1}x^{2k+1}f_{nT}(x|\mathbf{a}_{n})dx &=&2^{-2k-1}\sum_{j=0}^{k}%
\binom{2k+1}{k-j}t_{2j+1}(\mathbf{a}_{n}),
\end{eqnarray*}%
where $t_{k},$ $k\geq 0$ are given by (\ref{tt})
\end{corollary}

\begin{proof}
Both formulae are based on the formula $\int_{-1}^{1}T_{m}(x)f_{nT}(x|%
\mathbf{a}_{n})dx\allowbreak =\allowbreak t_{m}(\mathbf{a}_{n})/2$, for $%
m\geq 1$ and $1$ if $m\allowbreak =\allowbreak 0$. Now we use well known
formulae:%
\begin{eqnarray*}
x^{2k} &=&2^{1-2k}\sum_{j=1}^{k}\binom{2k}{k-j}T_{2j}(x)+\binom{2k}{k}%
2^{-2k}, \\
x^{2k+1} &=&2^{-2k}\sum_{j=0}^{k}\binom{2k+1}{k-j}T_{2j+1}(x).
\end{eqnarray*}
\end{proof}

In particular we have:

\begin{corollary}
\label{cases}Recall that $T_{2}(x)\allowbreak =\allowbreak 2x^{2}-1$ and
assume that $X\allowbreak \symbol{126}\allowbreak \allowbreak f_{nT}(x|%
\mathbf{a}_{n}),$ then $EX\allowbreak =\allowbreak t_{1}(\mathbf{a}_{n})/2,$ 
$\limfunc{var}(X)\allowbreak =\allowbreak 1/2\allowbreak +\allowbreak (t_{2}(%
\mathbf{a}_{n})/4-t_{1}^{2}(\mathbf{a}_{n})/4).$ If $n\allowbreak
=\allowbreak 1$we have $EX\allowbreak =\allowbreak a/2$ and $\limfunc{var}%
(X)\allowbreak =\allowbreak 1/2.$ If $n\allowbreak =\allowbreak 2$ we get $%
EX\allowbreak =\allowbreak (a_{1}+a_{2})/2,$ $\limfunc{var}(X)\allowbreak
=\allowbreak B_{2}(a_{1},a_{2})/2\allowbreak -\allowbreak \frac{%
a_{1}a_{2}(1-a_{1}a_{2})}{4(2-a_{1}a_{2})}.$
\end{corollary}

Following expansion (\ref{expan1}) we can get an expansion of orthogonal
(with respect to the measure $f_{nT}$) polynomials in series of polynomials $%
\left\{ T_{n}\right\} $ and also the 3 term formula satisfied by these
orthogonal polynomials. More precisely we have:

\begin{theorem}
Let us denote by $d_{0,0}\allowbreak =\allowbreak 1\allowbreak $ and for $%
m\geq 1:d_{m,i},$ $i\allowbreak =0\allowbreak ...,m-1$ solutions of the
following system of linear equations:%
\begin{gather}
\frac{1}{2}t_{m}+\frac{1}{2}\sum_{i=1}^{m-1}d_{m,i}t_{i}+d_{m,0}=0,
\label{w_ort0} \\
\frac{1}{2}(t_{m+j}+t_{m-j})+\frac{1}{2}\sum_{i=0,i\neq
j}^{m-1}d_{m,i}(t_{i+j}+t_{\left\vert i-j\right\vert
})+d_{m,j}(t_{2j}/2+1)=0,  \label{w_ortj}
\end{gather}%
for $j\allowbreak =\allowbreak 1,...,m-1,$ where $t_{j}\allowbreak
=\allowbreak t_{j}(\mathbf{a}_{n})$ are defined by (\ref{tt}). Then,
polynomials 
\begin{equation*}
p_{m}(x|\mathbf{a}_{n})\allowbreak =\allowbreak
T_{m}(x)+\sum_{i=0}^{m-1}d_{m,i}T_{i}(x),
\end{equation*}%
defined for $m\geq 0$ are orthogonal to with respect to the measure with the
density $f_{nT}$. Moreover, they satisfy the following $3$ term recurrence:%
\begin{gather}
p_{m+1}(x|\mathbf{a}_{n})-(2x+d_{m+1,m}-d_{m,m-1})p_{m}(x|\mathbf{a}%
_{n})\allowbreak  \label{3tr} \\
=\allowbreak (d_{m+1,m-1}-d_{m,m-2}-(d_{m+1,m}-d_{m,m-1})d_{m,m-1})p_{m-1}(x|%
\mathbf{a}_{n}).  \label{3tr1}
\end{gather}%
for $m\geq 1$ with $p_{-1}(x|\mathbf{a}_{n})\allowbreak =\allowbreak 0,$ $%
p_{0}(x|\mathbf{a}_{n})=1$.
\end{theorem}

\begin{proof}
Let us notice that for $p_{m}$ to be orthogonal with respect to $f_{nT}$ we
have to have for every $m\geq 1:$ $\int_{-1}^{1}T_{j}(x)p_{m}(x|\mathbf{a}%
_{n})f_{nT}(x|\mathbf{a}_{n})dx\allowbreak =\allowbreak 0$ for $j\allowbreak
=\allowbreak 0,...,m-1.$ Since we have $2T_{j}(x)T_{m}(x)\allowbreak
=\allowbreak T_{j+m}+T_{\left\vert j-m\right\vert },$ and (\ref{ortT}) we
directly get (\ref{w_ort0}) and (\ref{w_ortj}). Now we have: 
\begin{equation*}
2xp_{m}(x|\mathbf{a}_{n})\allowbreak =\allowbreak
T_{m+1}+T_{m-1}+\sum_{i=0,i\neq
1}^{m-1}d_{m,i}(T_{i+1}+T_{i-1})+d_{m,1}(T_{2}+1).
\end{equation*}%
Hence, $p_{m+1}(x|\mathbf{a}_{n})-2xp_{m}(x|\mathbf{a}_{n})\allowbreak
=\allowbreak
(d_{m+1,m}-d_{m,m-1})T_{m}+(d_{m+1,m-1}-d_{m,m-2})T_{m-1}\allowbreak
+\allowbreak $terms depending on $T_{i},$ $i\allowbreak =\allowbreak
0,...,m-2.$

Now notice that 
\begin{eqnarray*}
(d_{m+1,m}-d_{m,m-1})p_{m}\allowbreak &=&\allowbreak
(d_{m+1,m}-d_{m,m-1})T_{m}+(d_{m+1,m}-d_{m,m-1})d_{m,m-1}T_{m-1} \\
&&+\text{terms of lower order}
\end{eqnarray*}%
Since, certainly, defined above polynomials $p_{m}$ are orthogonal hence,
they satisfy some 3 term recurrence which is uniquely defined.
\end{proof}

Let us start with $n\allowbreak =\allowbreak 2.$ We have the following
observation:

\begin{corollary}
Density $f_{2T}(x|a_{1},a_{2})\allowbreak $ has the expansion:%
\begin{equation}
f_{2T}(x|a_{1},a_{2})\allowbreak =\allowbreak f_{C}(x)\sum_{k=0}^{\infty
}t_{k}(a_{1},a_{2})T_{k}(x),\allowbreak ,  \label{rozw}
\end{equation}%
where $B_{2}\allowbreak =\allowbreak \frac{2-2a_{1}a_{2}}{2-a_{1}a_{2}}$ and 
$t_{k}(a_{1},a_{2})\allowbreak =\allowbreak \frac{1}{2-a_{1}a_{2}}%
((1-a_{1}a_{2})\sum_{j=0}^{k}a_{1}^{j}a_{2}^{k-j}+a_{1}^{k}+a_{2}^{k})$
\end{corollary}

\begin{proof}
Following assertions of Theorem \ref{density} all one has to calculate are
coefficients $t_{j}(a_{1},a_{2}).$ Recall that:

\begin{equation*}
2t_{j}(a_{1},a_{2})=B_{2}(\frac{(2a_{1}-a_{2}-a_{1}^{2}a_{2})a_{1}^{j}}{%
(a_{1}-a_{2})(1-a_{1}a_{2})}+\frac{(2a_{2}-a_{1}-a_{2}^{2}a_{1})a_{2}^{j}}{%
(a_{2}-a_{1})(1-a_{1}a_{2})}).
\end{equation*}%
Following this we have:%
\begin{gather*}
2t_{j}(a_{1},a_{2})=B_{2}\frac{%
2(a_{1}^{j+1}-a_{2}^{j+1})-a_{1}a_{2}(a_{1}^{j-1}-a_{2}^{j-1})-a_{1}a_{2}(a_{1}^{j+1}-a_{2}^{j+1})%
}{(a_{1}-a_{2})(1-a_{1}a_{2})} \\
=B_{2}\frac{(2-a_{1}a_{2})}{1-a_{1}a_{2}}%
\sum_{k=0}^{j}a_{1}^{j-k}a_{2}^{k}-2B_{2}\frac{a_{1}a_{2}}{1-a_{1}a_{2}}%
\sum_{k=0}^{j-2}a_{1}^{j-2-k}a_{2}^{k} \\
=B_{2}\sum_{k=0}^{j}a_{1}^{j-k}a_{2}^{k}+B_{2}\frac{1}{1-a_{1}a_{2}}%
(\sum_{k=0}^{j}a_{1}^{j-k}a_{2}^{k}-%
\sum_{k=0}^{j-2}a_{1}^{j-1-k}a_{2}^{k+1})= \\
\frac{2(1-a_{1}a_{2})}{2-a_{1}a_{2}}\sum_{k=0}^{j}a_{1}^{j-k}a_{2}^{k}+\frac{%
2\left( a_{1}^{j}+a_{2}^{j}\right) }{2-a_{1}a_{2}}.
\end{gather*}
\end{proof}

\begin{remark}
To get expansion (\ref{expan1}) or to find coefficients (\ref{bb}), crucial
for the further calculations, we had to assume that all parameters $a_{i}$
are different. However notice that this is only technical, artificial
assumption. As far as (\ref{fnT}) with (\ref{fTg}) inserted is well defined
for all $\left\vert a_{i}\right\vert <1$ even all or some equal. To be able
to consider such case, one can pass with these $a_{j}$ to limits $a_{i}$ if
we are to have $a_{j}\allowbreak =\allowbreak a_{i}$ for some $i$ and $j$
where it is required.
\end{remark}

\subsection{Complex parameters\label{zesp}}

In this subsection we will consider complex parameters $a_{i},$ $%
i\allowbreak =\allowbreak 1,\ldots ,n.$ To do so we will assume that $n$ is
even and that the parameters $\left\{ a_{i}\right\} $ form conjugate pairs.
To fix notation let us assume that the first two parameters form a first
conjugate pair, the third and forth the second conjugate pair and so on. Let
us denote 
\begin{equation*}
f_{n}(x|\rho _{1},y_{1},\rho _{2},y_{2},\ldots ,\rho _{n},y_{n})\allowbreak
=\allowbreak f_{2nT}(x|\mathbf{a}_{2n}),
\end{equation*}%
where we denoted $a_{1}\allowbreak =\allowbreak \rho _{1}\exp (i\theta
_{1}),a_{2}\allowbreak =\allowbreak \rho _{1}\exp (-i\theta _{1}),\ldots
,a_{2n-1}\allowbreak =\allowbreak \rho _{n}\exp (i\theta
_{n}),a_{2n}\allowbreak =\allowbreak \rho _{n}\exp (-i\theta _{n}))$ and $%
\theta _{i}\allowbreak =\allowbreak y_{i}$, $i\allowbreak =\allowbreak
1,\ldots ,n$.

Notice that for the pair $a\allowbreak =\allowbreak \rho \exp (i\theta ),$ $%
b\allowbreak =\allowbreak \rho \exp (-i\theta )$ we have $%
(1+a^{2}-2ax)(1+b^{2}-2bx)\allowbreak =\allowbreak w(x,y|\rho )$ and $%
(1-ax)(1-bx)\allowbreak =\allowbreak 1\allowbreak -\allowbreak 2\rho
xy\allowbreak +\allowbreak \rho ^{2}x^{2},$ where as before $y\allowbreak
=\allowbreak \cos \theta .$

\begin{corollary}
$f_{2}(x|\rho ,y)\allowbreak =\allowbreak f_{C}(x)\frac{2(1-\rho ^{2})}{%
2-\rho ^{2}}\frac{1-2xy\rho +\rho ^{2}x^{2}}{w(x,y|\rho )}$ allows the
following expansion:%
\begin{equation}
f_{2}(x|\rho ,y)\allowbreak =\allowbreak f_{C}(x)+\frac{1-\rho ^{2}}{2-\rho
^{2}}f_{C}(x)\sum_{j\geq 1}\rho ^{j}T_{j}(x)U_{j}(y)+\frac{2}{2-\rho ^{2}}%
f_{C}(x)\sum_{j\geq 1}\rho ^{j}T_{j}(x)T_{j}(y).  \label{war}
\end{equation}
\end{corollary}

\begin{proof}
We use (\ref{rozw}). Next we observe: $\sum_{j=0}^{k}a_{1}^{j}a_{2}^{k-j}%
\allowbreak =\allowbreak \rho ^{k}\sum_{j=0}^{k}\exp (ij\theta )\exp
(-i(k-j)\theta \mathbb{)\allowbreak =\allowbreak }\rho ^{k}\exp (-ik\theta
)\sum_{j=0}^{k}\exp (2ij\theta )\allowbreak =\allowbreak \rho ^{k}\frac{\exp
(-ik\theta )(1-\exp (i(k+1)\theta )}{1-\exp (2i\theta )}\allowbreak
=\allowbreak \rho ^{k}\frac{2i\sin (k+1)\theta }{2i\sin \theta }\allowbreak
=\allowbreak \rho ^{k}U_{k}(y).$ Secondly we have: $a_{1}^{k}+a_{2}^{k}%
\allowbreak =\allowbreak 2\rho ^{k}\cos k\theta \allowbreak =\allowbreak
2\rho ^{k}T_{k}(y)$ and finally; $a_{1}a_{2}\allowbreak =\allowbreak \rho
^{2}.$
\end{proof}

\begin{corollary}
\label{1param_cond}i) $\int_{-1}^{1}f_{2}(x|\rho ,y)f_{C}(y)dy\allowbreak
=\allowbreak f_{C}(x)+\frac{1-\rho ^{2}}{2-\rho ^{2}}f_{C}(x)\sum_{j\geq
1}\rho ^{2j}T_{2j}(x)\allowbreak =\allowbreak f_{C}(x)\frac{2(1+\rho
^{2}-x^{2}\rho ^{2}(3-\rho ^{2}))}{(2-\rho ^{2})((1+\rho ^{2})^{2}-4\rho
^{2}x^{2})}.$

ii) $\int_{-1}^{1}f_{2}(x|\rho ,y)f_{W}(y)dy\allowbreak =\allowbreak
f_{C}(x)\allowbreak (1-\frac{\rho x}{2-\rho ^{2}})$
\end{corollary}

\begin{proof}
i) The first formula comes straightforwardly from (\ref{war}). The second
one comes from the application of (\ref{parz}).

ii) We utilize the fact that $\forall n\geq 2:\frac{2}{\pi }%
\int_{-1}^{1}U_{n}(y)\sqrt{1-y^{2}}dy\allowbreak =0$ and \newline
$\frac{2}{\pi }\int_{-1}^{1}T_{n}(y)\sqrt{1-y^{2}}dy\allowbreak =0$ while $%
\frac{2}{\pi }\int_{-1}^{1}U_{1}(y)\sqrt{1-y^{2}}dy\allowbreak =0$ and%
\newline
$\frac{2}{\pi }\int_{-1}^{1}T_{1}(y)\sqrt{1-y^{2}}dy\allowbreak =-\frac{1}{2}%
.$
\end{proof}

\begin{remark}
Notice that $f_{2}(x|y,\rho )$ as well as the densities exposed in Corollary %
\ref{1param_cond} are conditional densities. Hence in particular, for every $%
\left\vert y\right\vert \leq 1$ and $\left\vert \rho \right\vert <1$ we get $%
\int_{-1}^{1}f_{2}(x|y,\rho )dx\allowbreak =\allowbreak 1$ or for every $%
\left\vert \rho \right\vert <1$ we get $\int_{-1}^{1}f_{C}(x)\frac{2(1+\rho
^{2}-x^{2}\rho ^{2}(3-\rho ^{2}))}{(2-\rho ^{2})((1+\rho ^{2})^{2}-4\rho
^{2}x^{2})}dx\allowbreak =\allowbreak 1$ or $\int_{-1}^{1}f_{C}(x)(1-\frac{%
\rho x}{2-\rho ^{2}})dx\allowbreak =\allowbreak 1.$ Besides, shapes of $%
f_{2} $ for different $y$ and $\rho $ are very versatile. For example we
have plots of $f_{2}(x|y,\rho )$ for $\left\vert x\right\vert <1$ for
different values of $y$ and $\rho $%
\begin{equation*}
\FRAME{itbpFUX}{10.256cm}{6.8491cm}{0cm}{\Qcb{Plot \protect\ref{fig1}}}{\Qlb{%
fig1}}{Plot}{\special{language "Scientific Word";type "MAPLEPLOT";width
10.256cm;height 6.8491cm;depth 0cm;display "USEDEF";plot_snapshots
TRUE;mustRecompute FALSE;lastEngine "MuPAD";xmin "-1.5";xmax "1.5";xviewmin
"-1.00016978165385";xviewmax "1.00017020486185";yviewmin
"0.163268026310314";yviewmax "1.01519013352569";plottype 4;axesFont "Times
New Roman,12,0000000000,useDefault,normal";numpoints 100;plotstyle
"patch";axesstyle "normal";axestips FALSE;xis \TEXUX{x};var1name
\TEXUX{$x$};function \TEXUX{$\frac{1}{\pi
\sqrt{1-x^{2}}}\frac{24}{7}\frac{4-2x+x^{2}}{13-20x+16x^{2}}$};linecolor
"blue";linestyle 1;pointstyle "point";linethickness 3;lineAttributes
"Solid";var1range "-1.5,1.5";num-x-gridlines 400;curveColor
"[flat::RGB:0x000000ff]";curveStyle "Line";rangeset"X";function
\TEXUX{$\frac{224}{23\pi
\sqrt{1-x^{2}}}\frac{16-6x+9x^{2}}{85-300x+576x^{2}}$};linecolor
"green";linestyle 1;pointstyle "point";linethickness 2;lineAttributes
"Solid";var1range "-5,5";num-x-gridlines 100;curveColor
"[flat::RGB:0x00008000]";curveStyle "Line";function \TEXUX{$\frac{225}{17\pi
\sqrt{1-x^{2}}}\frac{25+8x+16x^{2}}{145+656x+1600x^{2}}$};linecolor
"maroon";linestyle 1;pointstyle "point";linethickness 2;lineAttributes
"Solid";var1range "-5,5";num-x-gridlines 100;curveColor
"[flat::RGB:0x00800000]";curveStyle "Line";VCamFile
'PD16UZ01.xvz';valid_file "T";tempfilename
'PD16UZ00.wmf';tempfile-properties "XPR";}}
\end{equation*}%
Here red plot is for $y\allowbreak =\allowbreak .5,$ $\rho \allowbreak
=\allowbreak .5,$ blue for $y\allowbreak =\allowbreak -1/5,$ $\rho
\allowbreak =\allowbreak 4/5,$ green for $y\allowbreak =\allowbreak 1/4,$ $%
\rho \allowbreak =\allowbreak 3/4.$ One can notice that for some values of $%
y $ and $\rho $ the plot has one maximum at some $x\in (-1,1).$
\end{remark}

Further, as $n>2$ is concerned we have:

\begin{proposition}
Let $n\allowbreak =\allowbreak 4.$ For $a_{1}\allowbreak =\rho _{1}\exp
(i\theta _{1}),$ $a_{2}\allowbreak =\allowbreak \rho _{1}\exp (-i\theta
_{1}),$ $a_{3}\allowbreak =\allowbreak \rho _{2}\exp (i\theta _{2}),$ $%
a_{4}\allowbreak =\allowbreak \rho _{2}\exp (-i\theta _{2})$ we get 
\begin{gather*}
B_{4}(\mathbf{a}_{4})\allowbreak =B_{4}(y_{1},\rho _{1},y_{2},\rho _{2}) \\
=\frac{\allowbreak 8(1-\rho _{1}^{2})(1-\rho _{2}^{2})w(y_{1},y_{2}|\rho
_{1}\rho _{2})}{\alpha (\rho _{1},\rho _{2})+4(\beta _{1}(\rho _{1},\rho
_{2})y_{1}^{2}+\beta _{2}(\rho _{1},\rho _{2})y_{2}^{2})-4\kappa (\rho
_{1},\rho _{2})y_{1}y_{2}}
\end{gather*}%
where $y_{1}\allowbreak =\allowbreak \cos (\theta _{1}),$ $y_{2}\allowbreak
=\allowbreak \cos (\theta _{2}),$ $\alpha \left( \rho _{1},\rho _{2}\right)
\allowbreak =\allowbreak (1-\rho _{1}^{2}\rho _{2}^{2})(4+(4-\rho
_{1}^{2}\rho _{2}^{2})(1-\rho _{1}^{2})(1-\rho _{2}^{2})),$ $\beta _{1}(\rho
_{1},\rho _{2})\allowbreak =\allowbreak \rho _{1}^{2}\rho _{2}^{2}(1-\rho
_{2}^{2})(2-\rho _{1}^{2}),$ $\beta _{2}(\rho _{1},\rho _{2})\allowbreak
=\allowbreak \rho _{1}^{2}\rho _{2}^{2}(1-\rho _{1}^{2})(2-\rho _{2}^{2}),$ $%
\kappa (\rho _{1},\rho _{2})\allowbreak =\allowbreak \rho _{1}\rho
_{2}((2+\rho _{1}^{2}\rho _{2}^{2})(1-\rho _{1}^{2})(1-\rho
_{2}^{2})+2(1-\rho _{1}^{2}\rho _{2}^{2})).$

Consequently 
\begin{equation}
f_{4}(x|y_{1},\rho _{1},y_{2},\rho _{2})=f_{C}(x)B_{4}(y_{1},\rho
_{1},y_{2},\rho _{2})\frac{(1-2\rho _{1}xy_{1}+\rho _{1}^{2}x^{2})}{%
w(x,y_{1}|\rho _{1})}\frac{(1-2\rho _{2}xy_{2}+\rho _{2}^{2}x^{2})}{%
w(x,y_{2}|\rho _{2})}  \label{f44}
\end{equation}
is a conditional density i.e. $\forall \left\vert y_{1}\right\vert
,\left\vert y_{2}\right\vert \leq 1,\left\vert \rho _{1}\right\vert
,\left\vert \rho _{2}\right\vert <1$%
\begin{equation*}
\int_{-1}^{1}f_{4}(x|y_{1},\rho _{1},y_{2},\rho _{2})dx=1.
\end{equation*}
\end{proposition}

\begin{proof}
We have $P_{4}(\mathbf{a}_{4})\allowbreak =\allowbreak (1-\rho
_{1}^{2})(1-\rho _{2}^{2})\allowbreak \times \allowbreak (1-\rho _{1}\rho
_{2}\exp (i(\theta _{1}-\theta _{2}))(1-\rho _{1}\rho _{2}\exp (-i(\theta
_{1}-\theta _{2}))\allowbreak \times \allowbreak (1-\rho _{1}\rho _{2}\exp
(i(\theta _{1}+\theta _{2}))(1-\rho _{1}\rho _{2}\exp (-i(\theta _{1}+\theta
_{2}))\allowbreak =\allowbreak (1-\rho _{1}^{2})(1-\rho _{2}^{2})\allowbreak
\times \allowbreak (1+\rho _{1}^{2}\rho _{2}^{2}-2\rho _{1}\rho _{2}\cos
(\theta _{1}-\theta _{2}))\allowbreak \times \allowbreak (1+\rho
_{1}^{2}\rho _{2}^{2}-2\rho _{1}\rho _{2}\cos (\theta _{1}+\theta
_{2}))\allowbreak =\allowbreak (1-\rho _{1}^{2})(1-\rho _{2}^{2})((1+\rho
_{1}^{2}\rho _{2}^{2})^{2}\allowbreak -\allowbreak 4y_{1}y_{2}\rho _{1}\rho
_{2}(1+\rho _{1}^{2}\rho _{2}^{2})\allowbreak +\allowbreak 2\rho
_{1}^{2}\rho _{2}^{2}(2y_{1}^{2}-1+2y_{2}^{2}-1)),$ since $\cos (\theta
_{1}-\theta _{2})\cos (\theta _{1}+\theta _{2})\allowbreak =\allowbreak
(\cos (2\theta _{1})+\cos (2\theta _{2}))/2$, $\cos (2\theta
_{1})\allowbreak =\allowbreak T_{2}(y_{1})\allowbreak =\allowbreak
2y_{1}^{2}-1,$ $\cos (\theta _{1}-\theta _{2})\allowbreak +\allowbreak \cos
(\theta _{1}+\theta _{2})\allowbreak =\allowbreak 2y_{1}y_{2}.$ Consequently 
$P_{4}(\mathbf{a}_{4})\allowbreak =\allowbreak (1-\rho _{1}^{2})(1-\rho
_{2}^{2})w(y_{1},y_{2},\rho _{1}\rho _{2}).$

$S_{4}(\mathbf{a}_{4})\allowbreak =\allowbreak \rho _{1}^{2}\rho _{2}^{2},$ $%
S_{3}(\mathbf{a}_{4})\allowbreak =\allowbreak \rho _{1}^{2}\rho _{2}\exp
(i\theta _{2})\allowbreak +\allowbreak \rho _{1}^{2}\rho _{2}\exp (-i\theta
_{2})\allowbreak +\allowbreak \rho _{2}^{2}\rho _{1}\exp (i\theta
_{1})\allowbreak +\allowbreak \rho _{2}^{2}\rho _{1}\exp (-i\theta
_{1})\allowbreak =\allowbreak 2\rho _{1}^{2}\rho _{2}y_{2}\allowbreak
+\allowbreak 2\rho _{2}^{2}\rho _{1}y_{1},$

$S_{2}(\mathbf{a}_{4})\allowbreak =\allowbreak \rho _{1}^{2}+\rho
_{2}^{2}\allowbreak +\allowbreak \rho _{1}\rho _{2}(\exp (i(\theta
_{1}+\theta _{2})+\exp (-i(\theta _{1}+\theta _{2})\allowbreak +\allowbreak
\exp (i(\theta _{1}-\theta _{2}))\allowbreak +\allowbreak \exp (i(\theta
_{1}-\theta _{2})))\allowbreak =\allowbreak \rho _{1}^{2}+\rho
_{2}^{2}\allowbreak +\allowbreak \rho _{1}\rho _{2}(2\cos (\theta
_{1}+\theta _{2})\allowbreak +\allowbreak 2\cos (\theta _{1}-\theta
_{2}))\allowbreak =\allowbreak \rho _{1}^{2}+\rho _{2}^{2}\allowbreak +4\rho
_{1}\rho _{2}y_{1}y_{2},$ $S_{1}(\mathbf{a}_{4})\allowbreak =\allowbreak
2\rho _{1}y_{1}\allowbreak +\allowbreak 2\rho _{2}y_{2}.$ Now we use
Corollary \ref{wyl} and get $P_{4}(\mathbf{a}_{4})\allowbreak +\allowbreak
7-3S_{2}(\mathbf{a}_{4})\allowbreak +\allowbreak S_{3}(\mathbf{a}_{4})(S_{1}(%
\mathbf{a}_{4})\allowbreak -\allowbreak S_{3}(\mathbf{a}_{4}))\allowbreak
-\allowbreak S_{4}(\mathbf{a}_{4})(4\allowbreak -\allowbreak 2S_{2}(\mathbf{a%
}_{4})\allowbreak +\allowbreak 3S_{4}(\mathbf{a}_{4})))\allowbreak
=\allowbreak (1-\rho _{1}^{2})(1-\rho _{2})w(y_{1},y_{2},\rho _{1}\rho
_{2})\allowbreak +\allowbreak 7\allowbreak -\allowbreak 3(\allowbreak \rho
_{1}^{2}+\rho _{2}^{2}\allowbreak +4\rho _{1}\rho _{2}y_{1}y_{2})\allowbreak
+\allowbreak 2\rho _{1}\rho _{2}(\rho _{1}y_{2}\allowbreak +\allowbreak \rho
_{2}y_{1})(2(\rho _{1}y_{1}\allowbreak +\allowbreak \rho
_{2}y_{2})\allowbreak -\allowbreak \rho _{1}^{2}+\rho _{2}^{2}\allowbreak
+4\rho _{1}\rho _{2}y_{1}y_{2})\allowbreak -\allowbreak \rho _{1}^{2}\rho
_{2}^{2}(4\allowbreak -\allowbreak 2(\rho _{1}^{2}+\rho _{2}^{2}\allowbreak
+4\rho _{1}\rho _{2}y_{1}y_{2})\allowbreak +\allowbreak 3\rho _{1}^{2}\rho
_{2}^{2}).$
\end{proof}

\begin{remark}
\label{f4}As one can check with a help of Mathematica this time for some
values of $y_{1},$ $y_{2},$ $\rho _{1},$ $\rho _{2}$ the plot of $f_{4}$ has
two maxima for some two values of $x\in (-1,1).$
\end{remark}

\section{Proofs\label{dow}}

\begin{proof}[Proof of the Proposition \protect\ref{aux}]
First let us show (\ref{expUT}). We will use (\ref{expU}) and $%
T_{n}(x)\allowbreak =\allowbreak U_{n}(x)-xU_{n-1}(x).$ We have:%
\begin{eqnarray*}
g_{TU}(x,y|\rho )\allowbreak &=&\allowbreak \sum_{j=0}^{\infty }\rho
^{j}U_{j}(x)T_{j}(y)\allowbreak =\allowbreak 1\allowbreak +\allowbreak
\sum_{j=1}^{\infty }\rho ^{j}U_{j}(x)(U_{j}(y)-yU_{j-1}(y))\allowbreak \\
&=&\allowbreak 1\allowbreak +\allowbreak \sum_{j=1}^{\infty }\rho
^{j}U_{j}(x)U_{j}(y)\allowbreak -\allowbreak y\sum_{j=1}^{\infty }\rho
^{j}U_{j}(x)U_{j-1}(y)\allowbreak =
\end{eqnarray*}%
and further 
\begin{eqnarray*}
&=&\frac{1-\rho ^{2}}{w(x,y|\rho )}-y\rho \sum_{j=0}^{\infty }\rho
^{j}U_{j+1}(x)U_{j}(y)\allowbreak = \\
&=&\allowbreak \frac{1-\rho ^{2}}{w(x,y|\rho )}-y\rho \sum_{j=0}^{\infty
}\rho ^{j}(2xU_{j}(x)-U_{j-1}(x))U_{j}(y)\allowbreak =\allowbreak \\
\allowbreak &=&\frac{1-\rho ^{2}}{w(x,y|\rho )}\allowbreak -\allowbreak
2xy\rho \frac{1-\rho ^{2}}{w(x,y|\rho )}+y\rho \sum_{j=1}^{\infty }\rho
^{j}U_{j-1}(x)U_{j}(y)\allowbreak
\end{eqnarray*}%
and finally%
\begin{eqnarray*}
&=&\allowbreak \frac{(1-\rho ^{2})(1-2xy\rho )}{w(x,y|\rho )}\allowbreak
+\allowbreak y\rho ^{2}\sum_{j=0}^{\infty }\rho ^{j}U_{j}(x)U_{j+1}(y)= \\
\allowbreak &=&\frac{(1-\rho ^{2})(1-2xy\rho )}{w(x,y|\rho )}\allowbreak
+\allowbreak y\rho ^{2}\sum_{j=0}^{\infty }\rho
^{j}U_{j}(x)(2yU_{j}(y)-U_{j-1}(y))\allowbreak \\
&=&\allowbreak \frac{(1-\rho ^{2})(1-2xy\rho )}{w(x,y|\rho )}\allowbreak
+\allowbreak 2y^{2}\rho ^{2}\frac{1-\rho ^{2}}{w(x,y|\rho )}\allowbreak
-\allowbreak y\rho ^{2}\sum_{j=1}^{\infty }\rho ^{j}U_{j}(x)U_{j-1}(y)
\end{eqnarray*}%
Hence, we have an equation:%
\begin{equation*}
\frac{1-\rho ^{2}}{w(x,y|\rho )}-yB=\frac{(1-\rho ^{2})(1-2xy\rho
+2y^{2}\rho ^{2})}{w(x,y|\rho )}-y\rho ^{2}B,
\end{equation*}%
where we denoted $B\allowbreak =\allowbreak B(x,y|\rho )\allowbreak
=\allowbreak \sum_{j=1}^{\infty }\rho ^{j}U_{j}(x)U_{j-1}(y).$ So 
\begin{equation*}
B(x,y|\rho )\allowbreak =\allowbreak \frac{2\rho (x-\rho y)}{w(x,y|\rho )}.
\end{equation*}%
and consequently 
\begin{equation*}
g_{TU}(x,y|\rho )\allowbreak \allowbreak =\allowbreak \frac{(1-\rho ^{2})}{%
w(x,y|\rho )}-y\frac{2\rho (x-\rho y)}{w(x,y|\rho )}\allowbreak =\allowbreak 
\frac{(1-\rho ^{2})-2\rho xy+2\rho ^{2}y^{2}}{w(x,y|\rho )}.
\end{equation*}%
Hence, we have shown (\ref{expUT}) and (\ref{n_n-1}). Using, well known
relationship that $T_{n}(x)\allowbreak =\allowbreak (U_{n}(x)\allowbreak
-\allowbreak U_{n-2}(x))/2,$ $n\geq 1$ we have 
\begin{gather}
g_{T}(x,y|\rho )=1\allowbreak +\allowbreak \frac{1}{4}\sum_{n\geq 1}\rho
^{n}(U_{n}(x)-U_{n-2}(x))(U_{n}(y)-U_{n-2}(y))  \label{*1} \\
=\frac{3}{4}\allowbreak +\allowbreak \frac{1}{4}g_{U}(x,y|\rho )\allowbreak
+\allowbreak \rho ^{2}g_{U}(x,y|\rho )\allowbreak /4-\allowbreak \frac{1}{4}%
\sum_{n\geq 1}\rho ^{n}(U_{n}(x)U_{n-2}(y)\allowbreak +\allowbreak
U_{n-2}(x)U_{n}(y)).  \label{*2}
\end{gather}%
$\allowbreak $Now we use formula $U_{n+1}(x)\allowbreak =\allowbreak
2xU_{n}(x)\allowbreak -\allowbreak U_{n-1}(x)$ getting:%
\begin{equation}
g_{T}(x,y|\rho )=\frac{(1+3\rho ^{2})}{4}g_{U}(x,y|\rho )\allowbreak
-\allowbreak \frac{1}{4}\rho ^{2}A\left( x,y|\rho \right) ,  \label{eqn2}
\end{equation}%
where $A\left( x,y|\rho \right) \allowbreak =\allowbreak x\sum_{n\geq
0}^{\infty }\rho ^{n}U_{n+1}(x)U_{n}(y)+y\sum_{n\geq 0}^{\infty }\rho
^{n}U_{n+1}(y)U_{n}(x).$ We calculate $A(x,y|\rho )$ using (\ref{expUT})
getting 
\begin{equation*}
A(x,y|\rho )\allowbreak =\allowbreak xB(x,y|\rho )/\rho +yB(y,x|\rho )/\rho =%
\frac{2(x^{2}+y^{2}-2\rho xy)}{w(x,y|\rho )}.
\end{equation*}%
Inserting in (\ref{eqn2}) we get (\ref{expT}). Thus, we have shown (\ref%
{expT}) and (\ref{A}).

Now to show (\ref{Tn_n-1}) we have denoting $A(x,y)\allowbreak
=\sum_{j=0}^{\infty }\rho ^{j}T_{j}(x)T_{j+1}(y):\allowbreak $%
\begin{eqnarray*}
A(x,y) &=&2y\sum_{j=0}^{\infty }\rho ^{j}T_{j}(x)T_{j}(y)-\sum_{j=0}^{\infty
}\rho ^{j}T_{j}(x)T_{j-1}(y) \\
&=&2y\sum_{j=0}^{\infty }\rho ^{j}T_{j}(x)T_{j}(y)-y-\rho A(y,x),
\end{eqnarray*}%
since $T_{-1}(y)\allowbreak =\allowbreak y.$ Now iterating we obtain
equation:%
\begin{equation*}
(1-\rho ^{2})A(x,y)\allowbreak =\allowbreak (y-\rho x)(2\sum_{j=0}^{\infty
}\rho ^{j}T_{j}(x)T_{j}(y)-1).
\end{equation*}%
Now we use (\ref{expT}). Finally to get (\ref{n_n-2}) we put $y\allowbreak
=\allowbreak x$ in equation (\ref{*1}, \ref{*2}).
\end{proof}

\begin{proof}[Proof of Lemma \protect\ref{num}]
iii) follows directly definition given by (\ref{bb}).

i) We justify it by induction. For $n\allowbreak =\allowbreak 1$ we
obviously have 
\begin{equation}
\frac{x-a}{x-b}\allowbreak =\allowbreak 1+\frac{b-a}{x-b}.  \label{fs}
\end{equation}
Now assuming that it is true for $n\allowbreak =\allowbreak m.$ For $%
n\allowbreak =\allowbreak m+1$ we have:%
\begin{gather*}
\prod_{i=}^{m}\frac{x-a_{i}}{x-b_{i}}(1+\frac{b_{m+1}-a_{m+1}}{x-b_{m+1}}%
)\allowbreak =\allowbreak 1+\frac{b_{m+1}-a_{m+1}}{x-b_{m+1}}\allowbreak \\
+\sum_{k=1}^{m}\frac{\prod_{i=1}^{m}(b_{k}-a_{i})}{(x-b_{k})\prod_{i=1,i\neq
k}^{m}(b_{k}-b_{i})}+\frac{b_{m+1}-a_{m+1}}{x-b_{m+1}}\sum_{k=1}^{m}\frac{%
\prod_{i=1}^{m}(b_{k}-a_{i})}{(x-b_{k})\prod_{i=1,i\neq k}^{m}(b_{k}-b_{i})}.
\end{gather*}%
Now notice that 
\begin{equation}
\frac{1}{(x-c)(x-d)}\allowbreak =\allowbreak \frac{1}{(x-c)(c-d)}\allowbreak
+\allowbreak \frac{1}{(x-d)(d-c)}.  \label{ss}
\end{equation}
We apply it with $c\allowbreak =\allowbreak b_{m+1}$ and $d\allowbreak
=\allowbreak b_{k}$. So we have:%
\begin{gather*}
\prod_{i=}^{m+1}\frac{x-a_{i}}{x-b_{i}}=1+\frac{b_{m+1}-a_{m+1}}{x-b_{m+1}}%
(1+\sum_{k=1}^{m}\frac{\prod_{i=1}^{m}(b_{k}-a_{i})}{(b_{m+1}-b_{k})%
\prod_{i=1,i\neq k}^{m}(b_{k}-b_{i})}) \\
+\sum_{k=1}^{m}\frac{\prod_{i=1}^{m}(b_{k}-a_{i})}{(x-b_{k})\prod_{i=1,i\neq
k}^{m}(b_{k}-b_{i})}+\sum_{k=1}^{m}\frac{%
\prod_{i=1}^{m}(b_{k}-a_{i})(b_{m+1}-a_{m+1})}{(b_{k}-b_{m+1})(x-b_{k})%
\prod_{i=1,i\neq k}^{m}(b_{k}-b_{i})}.
\end{gather*}%
Now notice that by induction assumption with $x\allowbreak =\allowbreak
b_{m+1}$ we have: 
\begin{equation*}
1\allowbreak +\allowbreak \sum_{k=1}^{m}\frac{\prod_{i=1}^{m}(b_{k}-a_{i})}{%
(b_{m+1}-b_{k})\prod_{i=1,i\neq k}^{m}(b_{k}-b_{i})}\allowbreak =\allowbreak
\prod_{i=1}^{m}\frac{b_{m+1}-a_{i}}{b_{m+1}-b_{i}}.
\end{equation*}
Hence we get 
\begin{equation*}
\frac{b_{m+1}-a_{m+1}}{x-b_{m+1}}(1+\sum_{k=1}^{m}\frac{%
\prod_{i=1}^{m}(b_{k}-a_{i})}{(b_{m+1}-b_{k})\prod_{i=1,i\neq
k}^{m}(b_{k}-b_{i})})\allowbreak =\allowbreak \frac{1}{x-b_{m+1}}\frac{%
\prod_{i=1}^{m+1}(b_{m+1}-a_{i})}{\prod_{i=1}^{m}(b_{m+1}-b_{i})}.
\end{equation*}%
Next we have 
\begin{gather*}
\sum_{k=1}^{m}\frac{\prod_{i=1}^{m}(b_{k}-a_{i})}{(x-b_{k})\prod_{i=1,i\neq
k}^{m}(b_{k}-b_{i})}\allowbreak +\allowbreak \sum_{k=1}^{m}\frac{%
\prod_{i=1}^{m}(b_{k}-a_{i})(b_{m+1}-a_{m+1})}{(b_{k}-b_{m+1})(x-b_{k})%
\prod_{i=1,i\neq k}^{m}(b_{k}-b_{i})}\allowbreak \\
=\allowbreak \sum_{k=1}^{m}\frac{\prod_{i=1}^{m}(b_{k}-a_{i})}{%
(x-b_{k})\prod_{i=1,i\neq k}^{m}(b_{k}-b_{i})}(1+\frac{(b_{m+1}-a_{m+1})}{%
(b_{k}-b_{m+1})})\allowbreak \\
=\allowbreak \sum_{k=1}^{m}\frac{\prod_{i=1}^{m+1}(b_{k}-a_{i})}{%
(x-b_{k})\prod_{i=1,i\neq k}^{m+1}(b_{k}-b_{i})}.
\end{gather*}
Hence we have proved i). Now notice that to get ii) we have to take in i) $%
a_{i}\allowbreak =\allowbreak 1/a_{i}$, $b_{i}\allowbreak =\allowbreak
(a_{i}+1/a_{i})/2.$ Then $b_{m}\allowbreak -\allowbreak a_{k}\allowbreak
=\allowbreak -\gamma (a_{m},a_{k})/(2a_{m}a_{k}),$ $b_{k}-b_{i}\allowbreak
=\allowbreak (a_{i}-a_{k})(1-a_{k}a_{i})/(2a_{k}a_{i})$ and $%
x-b_{i}\allowbreak =\allowbreak -(1+a_{i}^{2}-2a_{i}x)/2a_{i}.$ Finally we
have to divide right hand side by $2^{n}$.

To get iv) we write 
\begin{gather}
\prod_{j=1}^{m+1}\frac{(1-a_{j}x)}{(1+a_{j}^{2}-2a_{j}x)}=  \label{_1} \\
\frac{1-a_{m+1}x}{1+a_{m+1}^{2}-2a_{m+1}x}(\frac{1}{2^{m}}+\frac{1}{2^{m}}%
\sum_{k=1}^{m}b_{m,k}\frac{(1-a_{k}^{2})}{(1+a_{k}^{2}-2a_{k}x)}).
\end{gather}%
Then we apply (\ref{fs}) with $a\allowbreak =\allowbreak 1/a_{m+11}$ and $%
b\allowbreak =\allowbreak (1/a_{m+1}\allowbreak +\allowbreak a_{m+1})/2$ and
then (\ref{ss}) with $c\allowbreak =\allowbreak (1/a_{1}+a_{1})$ and $%
d\allowbreak =\allowbreak (1/a_{2}\allowbreak +\allowbreak a_{2})/2$ applied
successfully to pairs $(a_{1},a_{m+1}),$ $\ldots ,(a_{m},a_{m+1})$. Finally
we get%
\begin{gather*}
\prod_{j=1}^{m+1}\frac{(1-a_{j}x)}{(1+a_{j}^{2}-2a_{j}x)}=\frac{1}{2^{m+1}}+%
\frac{1}{2^{m+1}}\sum_{k=1}^{m}\frac{(1-a_{k}^{2})}{(1+a_{k}^{2}-2a_{k}x)}%
b_{m,k} \\
+\frac{(1-a_{m+1}^{2})}{2^{m+1}(1+a_{m+1}^{2}-2a_{m+1}x)} \\
+\frac{1}{2^{m+1}}\sum_{k=1}^{m}b_{m,k}(\frac{%
a_{k}(1-a_{k}^{2})(1-a_{m+1}^{2})}{%
(a_{k}-a_{m+1})(1-a_{k}a_{m+1})(1+a_{k}^{2}-2a_{k}x)} \\
+\frac{a_{m+1}(1-a_{k}^{2})(1-a_{m+1}^{2})}{%
(a_{m+1}-a_{k})(1-a_{k}a_{m+1})(1+a_{m+1}^{2}-2a_{m+1}x)}) \\
=\frac{1}{2^{m+1}}+\frac{1}{2^{m+1}}\sum_{k=1}^{m}b_{m,k}\frac{(1-a_{k}^{2})%
}{(1+a_{k}^{2}-2a_{k}x)}(1+\frac{a_{k}(1-a_{m+1}^{2})}{%
(a_{k}-a_{m+1})(1-a_{k}a_{m+1})}) \\
+\frac{1}{2^{m+1}}\frac{(1-a_{m+1}^{2})}{(1+a_{m+1}^{2}-2a_{m+1}x)}%
(1+\sum_{k=1}^{m}b_{m,k}\frac{a_{m+1}(1-a_{k}^{2})}{%
(a_{m+1}-a_{k})(1-a_{k}a_{m+1})})
\end{gather*}

Since we know that the formula ii) is true we deduce that we must have $%
1+\sum_{k=1}^{m}b_{m,k}\frac{a_{m+1}(1-a_{k}^{2})}{%
(a_{m+1}-a_{k})(1-a_{k}a_{m+1})}\allowbreak =\allowbreak b_{m+1,m+1}.$
\end{proof}

\begin{proof}[Proof of the Theorem \protect\ref{density}]
We use expansion (\ref{dec}) and recall (\ref{expU}). Hence we have the
following expansion: 
\begin{equation*}
f_{nT}(x|\mathbf{a}_{n})=B_{n}(\mathbf{a}_{n})f_{C}(x)(\frac{1}{2^{n}}+\frac{%
1}{2^{n}}\sum_{j=0}^{\infty
}U_{j}(x)\sum_{k=1}^{n}(1-a_{k}^{2})b_{n,k}a_{k}^{j}).
\end{equation*}%
where we denoted $b_{n,k}\allowbreak $are given by (\ref{bb}). Now recall
that 
\begin{equation*}
\int_{-1}^{1}U_{j}(x)f_{C}(x)dx=\left\{ 
\begin{array}{ccc}
0 & if & j\text{ is odd} \\ 
1 & if & j\text{ is even}%
\end{array}%
\right. .
\end{equation*}%
Hence, we have relationship $1\allowbreak =\allowbreak B_{n}(\mathbf{a}%
_{n})(1/2^{n}\allowbreak +\allowbreak (\sum
(1-a_{k}^{2})b_{n.k}/(1-a_{k}^{2}))/2^{n}).$

To get the second assertion we use expansion (\ref{_1}) and identity $\frac{%
1-a_{k}x}{(1+a_{k}^{2}-2a_{k}x)}-\frac{1}{2}\allowbreak =\allowbreak \frac{%
(1-a_{k}^{2})}{(1+a_{k}^{2}-2a_{k}x)}$. We get then:%
\begin{gather*}
f_{nT}(x|\mathbf{a}_{n})=B_{n}(\mathbf{a}_{n})f_{C}(x)(\frac{1}{2^{n}}+\frac{%
1}{2^{n}}\sum_{k=1}^{n}b_{n,k}\frac{(1-a_{k}^{2})}{(1+a_{k}^{2}-2a_{k}x)}= \\
B_{n}(\mathbf{a}_{n})f_{C}(x)(\frac{1}{2^{n}}+\frac{1}{2^{n-1}}%
\sum_{k=1}^{n}b_{n,k}(\frac{1-a_{k}x}{(1+a_{k}^{2}-2a_{k}x)}-\frac{1}{2})= \\
B_{n}(\mathbf{a}_{n})f_{C}(x)(\frac{1}{2^{n}}-\frac{1}{2^{n}}%
\sum_{k=1}^{n}b_{n,k}+\frac{1}{2^{n-1}}\sum_{k=1}^{n}b_{n,k}\frac{1-a_{k}x}{%
(1+a_{k}^{2}-2a_{k}x)})= \\
B_{n}(\mathbf{a}_{n})f_{C}(x)(\frac{1}{2^{n-1}}-\frac{1}{2^{n}}-\frac{1}{%
2^{n}}\sum_{k=1}^{n}b_{n,k}+\frac{1}{2^{n-1}}\sum_{k=1}^{n}b_{n,k}\frac{%
1-a_{k}x}{(1+a_{k}^{2}-2a_{k}x)}).
\end{gather*}%
Further recalling expansion (\ref{fTg}) we get 
\begin{gather*}
f_{nT}(x|\mathbf{a}_{n})=B_{n}(\mathbf{a}_{n})f_{C}(x)(\frac{1}{2^{n-1}}-%
\frac{1}{B_{n}(\mathbf{a}_{n})})+\frac{B_{n}(\mathbf{a}_{n})}{2^{n-1}}%
f_{C}(x)\sum_{k=1}^{n}b_{n,k}\sum_{j=0}^{\infty }a_{k}^{j}T_{j}(x) \\
=(B_{n}(\mathbf{a}_{n})/2^{n-1}-1)f_{C}(x)+f_{C}(x)\frac{B_{n}(\mathbf{a}%
_{n})}{2^{n-1}}\sum_{k=1}^{n}b_{n,k} \\
+\frac{B_{n}(\mathbf{a}_{n})}{2^{n-1}}f_{C}(x)\sum_{j=1}^{\infty
}T_{j}(x)\sum_{k=1}^{n}b_{n,k}a_{k}^{j} \\
=(B_{n}(\mathbf{a}_{n})/2^{n-1}-1)f_{C}(x)+f_{C}(x)B_{n}(\mathbf{a}_{n})(%
\frac{2}{B_{n}(\mathbf{a}_{n})}-\frac{1}{2^{n-1}}) \\
+\frac{B_{n}(\mathbf{a}_{n})}{2^{n-1}}f_{C}(x)\sum_{j=1}^{\infty
}T_{j}(x)\sum_{k=1}^{n}b_{n,k}a_{k}^{j}=f_{C}(x)+\frac{B_{n}(\mathbf{a}_{n})%
}{2^{n-1}}f_{C}(x)\sum_{j=1}^{\infty }T_{j}(x)\sum_{k=1}^{n}b_{n,k}a_{k}^{j}
\end{gather*}%
Finally we recall definition of $B_{n}(\mathbf{a}_{n})$.
\end{proof}

\end{document}